\documentclass[12pt,twoside,leqno]{article}
\usepackage{amsmath}
\usepackage{amssymb}
\usepackage{amsxtra}
\usepackage{amscd}
\usepackage{amsthm}
\usepackage[mathscr]{eucal}

\theoremstyle{plain}
\newtheorem{thm}[subsection]{Theorem}
\newtheorem{prop}[subsection]{Proposition}
\newtheorem{cor}[subsection]{Corollary}
\newtheorem{lem}[subsection]{Lemma}

\theoremstyle{definition}

\newtheorem{rem}[subsection]{Remark}
\newtheorem{para}[subsection]{}

\newenvironment{pf}{\proof[\proofname]}{\endproof}

\newcommand{\ul}[1]{\underline{#1}}

\newcommand{\Pic}{{\rm Pic}}
\newcommand{\Div}{{\rm Div}}

\newcommand{\Hom}{{\rm Hom}}
\newcommand{\Ext}{{\rm Ext}}
\newcommand{\im}{{\rm im}}

\newcommand{\Spec}{{\rm Spec}}

\newcommand{\Alb}{{\rm Alb}}
\newcommand{\gr}{{\rm gr}}
\newcommand{\Ker}{{\rm Ker}}
\newcommand{\Lie}{{\rm Lie}}
\newcommand{\infi}{{\rm inf}}
\newcommand{\et}{{\rm \acute{e}t}}

\newcommand{\modu}{{\rm mod}}
\newcommand{\Ab}[1]{{\mathcal A} {\mathit b}/#1}

\newcommand{\sC}{{\mathcal C}}
\newcommand{\sD}{{\mathcal D}}

\newcommand{\sF}{{\mathcal F}}

\newcommand{\sH}{{\mathcal H}}

\newcommand{\sM}{{\mathcal M}}

\newcommand{\sO}{{\mathcal O}}

\newcommand{\C}{{\mathbb C}}

\newcommand{\G}{{\mathbb G}}

\newcommand{\N}{{\mathbb N}}

\newcommand{\Q}{{\mathbb Q}}

\newcommand{\Z}{{\mathbb Z}}

\newcommand{\phe}{{\varphi}}

\newcommand{\bs}{{\backslash}}
\newcommand{\lra}{\longrightarrow}

\begin{document}


\centerline{ } 
\vspace{25pt} 
\centerline{\LARGE Albanese varieties with modulus and Hodge theory}
\vspace{20pt}
\centerline{\large Kazuya Kato and Henrik Russell} 
\vspace{10pt}
\centerline{\large May 2009} 
\vspace{20pt}

\section{Introduction}

\begin{para} 
Let $X$ be a proper smooth variety over a field $k$ of characteristic
$0$, and let $\Alb(X)$ be the Albanese variety of $X$. In the work
\cite{Ru}, the second author constructed generalized Albanese varieties $\Alb_{\sF}(X)$, which are commutative connected algebraic groups over $k$ with surjective
homomorphisms $\Alb_{\sF}(X)\to \Alb(X)$ (see section~\ref{sec:Alb_with_mod}
 for a review).
If $Y$ is an effective divisor on $X$, a special case of
$\Alb_{\sF}(X)$ becomes the generalized Albanese variety $\Alb(X, Y)$ of
$X$ of modulus $Y$ (cf. section~\ref{sec:Alb_with_mod}). This is a higher dimensional analogue of the
generalized Jacobian variety with modulus of Rosenlicht-Serre. Note that
the divisor $Y$ can have multiplicity, and so the algebraic group
$\Alb(X, Y)$ can have an additive part.

Assume now $k=\C$. 
The purpose of this paper is to give  Hodge theoretic presentations (Thm.~\ref{mainthm}) of $\Alb(X, Y)$.

The case $Y$ has no multiplicity was studied in the work \cite{BaS} of 
Barbieri-Viale and Srinivas. 
A Hodge theoretic presentation of a generalized Albanese variety 
in the case without modulus but allowing the singularity of $X$ 
was given in the work \cite{ESV} of Esnault, Srinivas and Viehweg. 
\end{para}

\begin{para}\label{curve} First we review the curve case. Let $X$ be a proper smooth curve over $\C$ and let $Y$ be an effective divisor on $X$. In this case, the Albanese variety $\Alb(X, Y)$ of $X$ relative to $Y$ coincides with the generalized Jacobian variety $J(X, Y)$ of $X$ relative to $Y$. In the following, we will write the complex analytic space associated to $X$ simply by $X$, and the sheaf of holomorphic functions on it by $\sO_X$. Let $I=\text{Ker}(\sO_X\to \sO_Y)$ be the ideal of $\sO_X$ which defines $Y$.
The cohomology below is for the topology of the analytic space $X$ (not for Zariski topology). 

The generalized Jacobian variety $J(X, Y)$  
 is
the kernel of the degree map $\text{Pic}(X, Y)\to \Z$
where $\text{Pic}(X, Y)=H^1(X, \text{Ker}(\sO_X^\times \to \sO_Y^\times))$. 
Let $j: X-Y\lra X$ be the inclusion map and let $j_!\Z(1)$ be the $0$-extension of the constant sheaf $\Z(1)$ of $X-Y$ to $X$. (For $r\in \Z$, $\Z(r)$ denotes $\Z (2\pi i)^r$ as usual.) Then we have an exact sequence
$$0 \lra j_!\Z(1) \lra I \overset{\exp}\longrightarrow \Ker(\sO_X^\times \to \sO_Y^\times)\lra 0$$
and hence we have an isomorphism 
\begin{equation} \text{Pic}(X, Y) \cong  H^2(X,
[j_!\Z(1) \to I]).\end{equation}
 Here in the complex $[j_!\Z(1) \to I]$, 
$j_!\Z(1)$ is put in degree $0$.

We have another presentation of $J(X, Y)$ given in (2) below. Let $I_1$ be the ideal of $\sO_X$ which defines the reduced part of $Y$ and let
$J=II_1^{-1}\subset \sO_X$. Note that the composition of the two inclusion maps of complexes
$$[I \overset{d}\to J\Omega^1_X] \lra [I \overset{d}\to \Omega^1_X]
\lra [I_1 \overset{d}\to \Omega^1_X]$$
is a quasi-isomorphism. Hence we have an isomorphism in the derived category
$$ [I \overset{d}\to \Omega^1_X]\cong [I_1 \overset{d}\to \Omega^1_X]\oplus
(\Omega^1_X/J\Omega^1_X)[-1].$$
Since $j_!\C\lra  [I_1 \overset{d}\to \Omega^1_X]$ is a quasi-isomorphism, we have
an exact sequence 
\begin{equation} H^0(X, \Omega^1_X) \lra H^1_c(X-Y, \C/\Z(1))\oplus H^0(X, \Omega^1_X/J\Omega^1_X)\lra J(X, Y) \lra 0.\end{equation}
(Here $H_c$ is the cohomology with compact supports.)

\end{para}

\begin{para} Now let $X$ be a proper smooth variety over $\C$ of dimension $n$ and let $Y$ be an effective divisor on $X$.

 Again in the following theorem, cohomology groups are for the topology of the complex analytic spaces, and the notation $\sO$ and $\Omega$ stand for analytic sheaves.

Let $I$ be the ideal of $\sO_X$ which defines $Y$, let $I_1$ be the ideal of $\sO_X$ which defines the reduced part of $Y$, and let $J=II_1^{-1}\subset \sO_X$.

\end{para}

\begin{thm}\label{mainthm} {\rm(1)} 
We have an exact sequence  $$0\lra \Alb(X, Y)\lra
H^{2n}(X, \sD_{X,Y}(n))\overset{\deg}\longrightarrow  \Z\lra 0,$$ 
where for $r\in \Z$,  $\sD_{X,Y}(r)$ denotes the kernel of the surjective homomorphism of complexes $\sD_X(r)\to \sD_Y(r)$
with $D_X(r)$ the Deligne complex
$$[\Z(r) \to  \sO_X \overset{d}\to \Omega_X^1  \overset{d}\to\ldots \overset{d}\to 
\Omega_X^{r-1}]$$
and $\sD_Y(r)$ the similar complex
$$[\Z(r)_Y \to  \sO_Y \overset{d}\to \Omega_Y^1  \overset{d}\to\ldots \overset{d}\to 
\Omega_Y^{r-1}].$$
\medskip

{\rm(2)} We have an exact sequence
$$H^{n-1}(X, \Omega_X^n) \lra H^{2n-1}_c(X-Y, \C/\Z(n)) \oplus H^{n-1}(X, \Omega^n_X/J\Omega_X^n) \lra \Alb(X, Y) \lra 0.$$

\medskip

\end{thm}

Note that the case $n=1$ of Thm.~\ref{mainthm} (1) (resp. (2)) becomes the presentation 
of $J(X, Y)$ given by (1) (resp. (2)) in \ref{curve}. 

\medskip

\begin{rem}\label{1.5} We give some remarks on this theorem.

\medskip

(a) The case $Y=0$ of Thm.~\ref{mainthm} (1) is nothing but the well known exact sequence  
\begin{equation}
\quad 0 \lra \Alb(X) \lra H^{2n}(X, \sD_X(n)) \overset{\deg}\longrightarrow  \Z\lra 0
\end{equation} 
by using the Deligne cohomology $H^{2n}(X, \sD_X(n))$. 
(Usually the Deligne cohomology 
$H^m(X, \sD_X(r))$ is denoted by $H^m_D(X, \Z(r))$.)
\medskip

The case $Y=0$ of Thm.~\ref{mainthm} (2) is nothing but the usual presentation 
\begin{equation} 
\Alb(X) \cong H_\Z\bs H_\C/F^0H_\C
\end{equation}
 of the Albanese variety 
$\Alb(X)$ of $X$, where $(H_\Z, H_\C, F^{\bullet})$
 is the following Hodge structure of weight $-1$. 
 $H_\Z=H^{2n-1}(X, \Z(n))/(\text{torsion part})$, 
 $H_\C=\C\otimes_{\Z} H_\Z=H^{2n-1}(X, \Omega_X^{\bullet})$, and $F^{\bullet}$ is the Hodge filtration on $H_\C$ defined as
$$F^{-1}=H_\C, \quad F^0=H^{n-1}(X, \Omega^n_X), \quad F^1=0.$$

\medskip

(b) Recall that the presentations (3) and (4) of
$\Alb(X)$ are related as follows. Consider the exact sequence of complexes
$0 \to \Omega_X^{\leq n-1}[-1] \to \sD_X(n) \to \Z(n) \to 0$, where $\Omega_X^{\leq n-1}$ denotes the part of degree $\leq n-1$ of the de Rham complex $\Omega_X^{\bullet}$, which is actually a quotient complex of $\Omega_X^{\bullet}$. By taking the cohomology associated to this exact sequence, we have an exact sequence
$$H^{2n-1}(X, \Z(n)) \lra H^{2n-1}(X, \Omega_X^{\leq n-1})\lra H^{2n}_D(X, \Z(n))
\overset{\deg}\longrightarrow   \Z\lra 0.$$
Since 
$$H^{2n-1}(X, \Omega_X^{\leq n-1})\cong H^{2n-1}(X, \Omega_X^{\bullet})/H^{n-1}(X, \Omega^n_X) \cong H^{2n-1}(X, \C)/H^{n-1}(X, \Omega^n_X),$$
the exact sequence (4) is equivalent to (3).

\medskip

(c) (1) and (2) of Thm.~\ref{mainthm} are related similarly. 
Let $S$ be the subcomplex of the de Rham complex 
$\Omega_X^{\bullet}$ of $X$ defined by 
$S^p=\Ker(\Omega_X^p\to \Omega_Y^p)$ for $0\leq p\leq n-1$ 
and $S^n=\Omega_X^n$. 
Then Thm.~\ref{mainthm} (1) is equivalent to
$$\Alb(X, Y) \cong H_\Z\bs H^{2n-1}(X, S)/H^{n-1}(X, \Omega_X^n)$$ 
where \;\; $H_\Z= H^{2n-1}_c(X-Y, \Z(n))/(\text{torsion part})$. 
As is shown in \S~\ref{proof}, we have a commutative diagram with an isomorphism in the lower row
$$\begin{matrix}
H^{n-1}(X, \Omega_X^n) & = & H^{n-1}(X, \Omega_X^n) \hspace{20pt} \\
\downarrow & & \downarrow \hspace{15pt} \\
H^{2n-1}(X, S) & \cong & H_c^{2n-1}(X-Y, \C)\oplus H^{n-1}(X, \Omega^n_X/J\Omega_X^n).
\end{matrix}$$
Thus (1) and (2) of Thm.~\ref{mainthm} are deduced from each other.

\end{rem}

\begin{para}\label{1.6} As mentioned above, Thm.~\ref{mainthm}
shows that $\Alb(X, Y)$ is expressed as 
$H_\Z\bs H_V/F^0$ where:

\medskip

 $H_\Z=H^{2n-1}_c(X-Y, \Z(n))/(\text{torsion part})$,

\medskip

$H_V=H_\C\oplus H^{n-1}(X, \Omega^n_X/J\Omega_X^n)\cong H^{2n-1}(X, S)$

($H_\C=\C \otimes H_\Z$ and 
  $S$ is as in 1.5 (d)),

\medskip
$F^{\bullet}$ is the decreasing filtration on $H_V$ given by
$$F^{-1}=H_V, \quad F^0=H^{n-1}(X, \Omega^n_X), \quad F^1=0.$$

\medskip

Note that $H_V$ can be different from $H_\C$ here, and so $(H_\Z, H_V, F^{\bullet})$ here need not be a Hodge structure. It is some kind of 
``mixed Hodge structure with additive part''. 
 This object $(H_\Z, H_V, F^{\bullet})$ with a weight filtration, which we will denote by $H^{2n-1}(X, Y_-)(n)$ in section~\ref{proof}, 
belongs to a category $\sH$ introduced in section~\ref{Hodge} which contains the category of mixed Hodge structures but is bigger than it. In the proof of Thm.~\ref{mainthm}, it is essential to consider such object. This category $\sH$ is related to the category of refined Hodge structures of Bloch-Srinivas \cite{BlS} and to the category of formal Hodge structures of Barbieri-Viale \cite{BV}. In the proof of Thm.~\ref{mainthm}, we use the result of Barbieri-Viale in \cite{BV} on the Hodge theoretic description of the category of ``$1$-motives with additive parts'' over $\C$.
\end{para}

\begin{para} The theory of generalized Albanese variety in characteristic $p>0$ is given in 
\cite{Ru3} basing on the duality theory \cite{Ru2} of $1$-motives with unipotent parts in characteristic $p>0$.

In characteristic $p>0$, the syntomic cohomology is an analogue of the Deligne cohomology. 
 We expect that we can have presentations 
of the $p$-adic completion of $\text{Alb}(X, Y)(k)$ ($k$ is the base field), which is similar to Thm.~\ref{mainthm}, 
by using crystalline cohomology theory and syntomic cohomology theory.
\end{para}

We are thankful to Professor H\'el\`ene Esnault for advice.
\medskip

\section{Mixed Hodge structures with additive  parts }\label{Hodge}

\medskip
\begin{para} For a proper smooth variety $X$ over $\C$ of dimension $n$ and for an effective divisor $Y$ on $X$, we will have in section~\ref{proof} 
certain structures $H^1(X, Y_+)$ and $H^{2n-1}(X, Y_-)$ which are kinds of ``mixed Hodge structures with additive parts''. (These structures for the case $X$ is a curve are explained in \ref{example} below.) The authors imagine that there is a nice definition of the category of ``mixed Hodge structures with additive parts'', which contains these $H^1(X,Y_+)$ and $H^{2n-1}(X, Y_-)$ as objects, but can not define it.  Instead, we define
a category $\sH$ having these objects, which may be a very simple approximation of such nice category. 
\end{para}

\begin{para}
The category $\sH$. An object of $\sH$ is $H=(H_\Z, H_V, W_{\bullet}H_\Q, W_{\bullet}H_V, F^{\bullet}H_V, a, b)$, 
where $H_\Z$ is a finitely generated $\Z$-module, $H_V$ is a finite 
dimensional $\C$-vector space, $W_{\bullet}H_\Q$ is an increasing filtration on $H_\Q:=\Q\otimes H_\Z$ (called weight filtration), $W_{\bullet}H_V$ is an increasing filtration on $H_V$ (called weight filtration), $F^{\bullet}$ is a decreasing filtration on $H_V$ (called Hodge filtration), $a$ is a $\C$-linear map $H_\C:=\C\otimes H_\Z\to H_V$ which sends $W_wH_\C:=\C\otimes_{\Q} W_wH_\Q$ into 
$W_wH_V$ for any $w\in \Z$, and $b$ is a $\C$-linear map $H_V\to  H_\C$ which sends $W_wH_V$ into $W_wH_\C$ for any $w\in \Z$ such that $b\circ a$ is the identity map of $H_\C$.

We sometimes denote an object $H$ of $\sH$ simply as $(H_\Z, H_V)$.

The category of mixed Hodge structures is naturally embedded in $\sH$ as a full subcategory, by putting $H_V=H_\C$. 

The full subcategory of $\sH$ consisting of all objects $H$ such that $H_\Z$ are torsion free is clearly self-dual.

This category $\sH$ is a simpler version of the category of 
enriched Hodge structures in Bloch-Srinivas \cite{BlS}.

\end{para}

\begin{para}\label{example} {\bf Example.}
Let $X$ be a proper smooth curve over $\C$ and let $Y$ be an effective divisor on $X$. Let $I$ be the ideal of $\sO_X$ which defines $Y$, let $I_1$ be the ideal of $\sO_X$ which defines the reduced part of $Y$, and let $J=II_1^{-1}\subset \sO_X$. 

We define objects $H^1(X, Y_+)$ and $H^1(X, Y_-)$ of $\sH$. 

First, we define $H=H^1(X, Y_+)$. Let $$H_\Z=H^1(X-Y, \Z), \quad
H_V=H^1(X, [\sO_X \overset{d}\to I^{-1}\Omega_X^1]).$$
The map $a: H_\C\to H_V$ is 
$$H^1(X-Y, \C) \cong H^1(X, [\sO_X\to I_1^{-1}\Omega_X^1])\lra H^1(X, [\sO_X\to I^{-1}\Omega_X^1]).$$ 
The map $b: H_V\to H_\C$ is the composition
$$H^1(X, [\sO_X\to I^{-1}\Omega_X^1])\lra 
H^1(X, [J^{-1}\to I^{-1}\Omega_X^1])
\overset{\simeq}\longleftarrow
 H^1(X, [\sO_X\to I_1^{-1}\Omega_X])$$ 
\hfill $\cong H^1(X-Y, \C).$\\
The weight filtrations and the Hodge filtration are given by
$$W_2H_\Q=H_\Q, \quad W_1H_\Q=H^1(X, \Q),  \quad W_0H_\Q=0,$$
$$W_2H_V=H_V, \quad W_1H_V=H^1(X, \C),\quad W_0H_V=0,$$
where $H^1(X, \C)$ is embedded in $H_V$ via $a$, and 
$$F^0H_V=H_V, \quad F^1H_V=H^1(X, \C), \quad F^2H_\C=0.$$

Next, we define $H=H^1(X, Y_-)$. Let $$H_\Z=H_c^1(X-Y, \Z), \quad
H_V=H^1( X, [I\overset{d}\to \Omega_X^1]).$$
The map $a: H_\C\to H_V$ is the composition 
$$H^1_c(X-Y, \C) \cong H^1(X, [I_1\to \Omega_X^1])
\overset{\simeq}\longleftarrow  H^1(X, [I\to J\Omega_X^1])\lra 
H^1(X, [I\to \Omega_X^1]).$$ 
The map $b: H_V\to H_\C$ is  
$$H^1(X, [I\to\Omega_X^1])\lra H^1(X, [I_1\to \Omega_X^1])
\cong H_c^1(X-Y, \C).$$
The weight filtrations and the Hodge filtration are given by
$$W_1H_\Q=H_\Q, \quad W_0H_Q=\Ker(H_\Q\to H^1(X, \Q),  \quad W_{-1}H_\Q=0,$$
$$W_1H_V=H_V, \quad W_0H_V=\Ker(H_V\to H^1(X, \C)),\quad W_{-1}H_V=0,$$
where $H^1(X, \C)$ is regarded as quotient of $H_V$ via $b$, and 
$$F^0H_V=H_V, \quad F^1H_V=\Ker(H_V \to H^1(X, \sO_X)), \quad 
F^2H_\C=0.$$

Then we have exact sequences  
in $\sH$
$$0 \lra H^1(X) \lra H^1(X, Y_+) \lra H^0(Y)(-1) \lra \Z(-1)\lra 0,$$
$$0\lra \Z\lra H^0(Y) \lra H^1(X, Y_-)\lra H^1(X)\lra 0.$$
Here for $r\in \Z$, $\Z(r)$ is the usual Hodge structure $\Z(r)$ regarded as an object of $\sH$. $H^1(X)$ is also the usual Hodge structure of weight $1$ associated to the first cohomology of $X$, regarded as an object of $\sH$. Finally the object $H^0(Y)$ of $\sH$ is defined as below, and $H^0(Y)(-1)$ is the $-1$ Tate twist. 

The definition of $H=H^0(Y)$ is as follows. $H_\Z=H^0(Y, \Z)=\oplus_{y\in Y} \Z$. $H_V=H^0(Y, \sO_Y)$. $a$ is the canonical map $H^0(Y, \C)\to H^0(Y, \sO_Y)$. $b$ is the canonical map $H^0(Y, \sO_Y)\to H^0(Y, \C)$ given by $\sO_Y\to\C$ which is $\sO_{Y, y} \to \sO_{Y, y}/m_y=\C$ at each $y\in Y$ ($m_y$ denotes the maximal ideal of $\sO_{Y, y}$). The weight filtration 
and the Hodge filtration are given by
$$W_0H=H, \quad W_{-1}H=0,$$
$$F^0H_V=H_V, \quad F^1H_V=0.$$
 Note that $H_\C\to 
H_V$ can be like $\C\to \C[T]/(T^n)$, and need not be an isomorphism. 

The evident self-duality $\underline{\Hom}(\phantom{a}, \Z)$ 
for torsion free objects in $\sH$ induces
$$H^1(X, Y_-) \cong \underline{\Hom}(H^1(X, Y_+), \Z)(-1).$$

\end{para}

\medskip

\section{$1$-motives with additive parts}\label{1mot}

In \cite{La}, Laumon formulated the notion ``$1$-motive with additive part'' 
over a field of characteristic $0$. We review it assuming that the base field  is algebraically closed for simplicity. 

Fix an algebraically closed field $k$ of characteristic $0$.

\begin{para}
Let $\Ab{k}$ be the category of sheaves of abelian groups on the 
fppf-site of the category of affine schemes over $k$. 
Let $\sC^{[-1,0]}(\Ab{k})$ be the abelian category of complexes in $\Ab{k}$ concentrated in degrees $-1$ and $0$. 

A \emph{$1$-motive with additive part} over $k$ is an object of 
$\sC^{[-1,0]}(\Ab{k})$ of the form $[\sF\to G]$, 
where $G$ is a commutative connected algebraic group over $k$ and $\sF\cong \Z^t\oplus (\widehat \G_a)^s$ for some $t$ and $s$. 
(Cf.\ \cite[Def.\ (5.1.1)]{La}.) 
Here $\Z$ is regarded as a constant sheaf and $\widehat \G_a$ 
denotes the formal completion of the additive group $\G_a$ at $0$. 
Recall that for any commutative ring $R$, $\widehat \G_a(R)$ is the subgroup of the additive group $R$ consisting of all nilpotent elements. We have $\sF=\sF_\et\oplus \sF_{\infi}$ where $\sF_\et$ is the \'etale part of $\sF$ which corresponds to  $\Z^t$ in the above isomorphism and $\sF_{\infi}$ is the infinitesimal part of $\sF$ which corresponds to $(\widehat \G_a)^s$.

We denote the category of $1$-motives with additive parts over $k$ by $\sM_1$.

\end{para}

\begin{para}\label{dual}
The category $\sM_1$ admits a notion of duality (called ``Cartier duality''). 
Let $[\sF\to G]$ be a $1$-motive with additive part over $k$. Then we have the ``Cartier dual'' $[\sF^*\to G^*]$ of $[\sF\to G]$ which is an object of $\sM_1$ obtained as follows. Let
$0 \to L \to G \to A \to 0$ be the canonical decomposition of 
$G$ as an extension of an abelian variety $A$ by a commutative connected affine algebraic group $L$. Note that $L\cong (\G_m)^t 
\oplus (\G_a)^s$ for some $t$ and $s$. 
We have $$\sF^*= \ul{\Hom}_{\Ab{k}}(L,\G_m), \quad G^*=\ul{\Ext}^1_{\sC^{[-1,0]}(\Ab{k})}([\sF\to A],\G_m)$$
and the homomorphism $\sF^*\to G^*$ is 
the connecting homomorphism 
$$\ul{\Hom}_{\Ab{k}}(L,\G_m) \lra 
\ul{\Ext}^1_{\sC^{[-1,0]}(\Ab{k})}([\sF\to A],\G_m)$$ 
associated to the short exact sequence $0 \lra L \lra [\sF\to G] \lra [\sF\to A] \lra 0$ in $\sC^{[-1,0]}(\Ab{k})$. 
Since $$\ul{\Hom}_{\Ab{k}}(\G_m,\G_m)\cong \Z,
 \quad \ul{\Hom}_{\Ab{k}}(\G_a,\G_m)\cong 
 \widehat \G_a,$$ we have $\sF^*\simeq \Z^t\oplus (\widehat \G_a)^s$ for some $t$ and $s$. We have an exact sequence 
 $$0 \lra \ul{\Hom}_{\Ab{k}}(\sF, \G_m) \lra \ul{\Ext}^1_{\sC^{[-1,0]}(\Ab{k})}([\sF\to A],\G_m)\lra  \ul{\Ext}^1_{\Ab{k}}(A,\G_m)\lra 0,$$
 $\ul{\Ext}^1_{\Ab{k}}(A,\G_m)$ is the dual abelian variety
 of $A$, 
 and since $$\ul{\Hom}_{\Ab{k}}(\Z, \G_m)\cong \G_m, \quad \ul{\Hom}_{\Ab{k}}(\widehat \G_a, \G_m)\cong \G_a,$$
 $\ul{\Hom}_{\Ab{k}}(\sF, \G_m)\cong (\G_m)^t \oplus (\G_a)^s$ for some $t$ and $s$. Hence $G^*$ is a commutative connected algebraic group over $k$. Thus 
 $[\sF^*\to G^*]$ is a $1$-motive with additive part. The Cartier dual of
 $[\sF^*\to G^*]$ is canonically isomorphic to $[\sF\to G]$.

See \cite[section 5]{La} for details or \cite[section 1]{Ru} for a review.
\end{para}

\begin{para} Let $\sM_{1,\{-1, -2\}}$ be the full subcategory of $\sM_1$ consisting of all objects $[\sF\to G]$ such that $\sF=0$.

Let $\sM_{1, \{0, -1\}}$ be the full subcategory of $\sM_1$ consisting of all objects $[\sF\to G]$ such that $G$ is an abelian variety. 

Then the self-duality of $\sM_1$ in \ref{dual} induces an anti-equivalence between the categories 
$\sM_{1,\{-1, -2\}}$ and $\sM_{1, \{0, -1\}}$.
\end{para}

\section{Equivalences of categories}\label{equiv}

In \cite{BV}, Barbieri-Viale constructed a Hodge theoretic category and proved that in the case the base field is $\C$,  the category $\sM_1$ is equivalent to his Hodge theoretic category. 
Here we reformulate his equivalence in the style which is convenient for us, by using the category $\sH$ from section~\ref{Hodge}.
 
\begin{para}
The category $\sH_1$. An object of $\sH_1$ is an object $H$ of $\sH$ endowed with a splitting of the weight filtration on $\Ker(H_V\to H_\C)$
satisfying the following conditions (i)--(iv).

\medskip
(i) $H_\Z$ is torsion free,  $F^{-1}H_V=H_V$, $F^1H_V=0$, $W_0H=H$, $W_{-3}H=0$. 

\medskip
(ii) $\gr^W_{-1}H$ is a polarizable Hodge structure of weight $-1$. That is, $\gr^W_{-1}H_\C=\gr^W_{-1}H_V$ and $\gr^W_{-1}H_\Z$ with the Hodge filtration on $\gr^W_{-1}H_\C$ is a polarizable Hodge structure of weight $-1$. 

\medskip

(iii) $F^0\gr^W_0H_V=\gr^W_0H_V$.

\medskip

(iv) $F^0W_{-2}H_V=0$.

\medskip

Morphisms of $\sH_1$ are the evident ones. 

The category $\sH_1$ is self-dual by the functor $\underline{\Hom}(\phantom{a}, \Z)(1)$. 
\end{para}

\begin{para} For a subset $\Delta$ of $\{0, -1, -2\}$, let $\sH_{1,\Delta}$ be the full subcategory of $\sH_1$ consisting of all objects $H$ such that $\gr^W_wH=0$ unless $w\in \Delta$.

The categories $\sH_{1,\{-1, -2\}}$ and $\sH_{1, \{0, -1\}}$ are important for us. These categories are in fact defined  as full subcategories of $\sH$ without reference to the splitting of the weight filtration on $\Ker(H_V\to  H_\C)$, for the weight filtrations on $\Ker(H_V\to  H_\C)$ of objects of these categories are pure. 

Thus $\sH_{1, \{-1, -2\}}$ is the full subcategory of $\sH$ consisting of all objects $H$ satisfying the following conditions (i)--(iii). 

\medskip

(i) $H_\Z$ is torsion free,  $F^{-1}H_V=H_V$, $F^1H_V=0$, 
$W_{-1}H=H$, $W_{-3}H=0$. 

\medskip

(ii) $\gr^W_{-1}H$ is a polarizable Hodge structure of weight $-1$.

\medskip

(iii) $F^0W_{-2}H_V=0$.

\medskip

For example, the Tate twist  $H^1(X, Y_-)(1)$ of the object $H^1(X, Y_-)$ of $\sH$ in \ref{example} belongs to $\sH_{1, \{-1,-2\}}$.
\medskip

Similarly, $\sH_{1, \{0, -1\}}$ is the full subcategory of $\sH$ consisting of all objects $H$ satisfying the following conditions (i)--(iii). 

\medskip
(i)  $H_\Z$ is torsion free, $F^{-1}H_V=H_V$, $F^1H_V=0$, $W_0H=H$, $W_{-2}H=0$. 

\medskip
(ii) $\gr^W_{-1}H$ is a polarizable Hodge structure of weight $-1$.

\medskip

(iii) $F^0\gr^W_0H_V=\gr^W_0H_V$.

\medskip

For example, the Tate twist  $H^1(X, Y_+)(1)$ of the object $H^1(X, Y_+)$ of $\sH$ in \ref{example} belongs to $\sH_{1, \{0,-1\}}$.
\medskip

The self-duality $\underline{\Hom}(\phantom{a}, \Z)(1)$ of $\sH_1$ induces an anti-equivalence between the categories 
 $\sH_{1, \{-1,-2\}}$ and $\sH_{1,\{0,-1\}}$.
\end{para}

\begin{thm}\label{thm2} (This is a reformulation of the equivalence of categories proved by Barbieri-Viale in \cite{BV}.)
We have an equivalence of categories $\sH_1\simeq \sM_1$ which is compatible with the dualities, and which induces the equivalences  
$$\sH_{1,\{-1, 0\}}\simeq \sM_{1,\{-1,0\}}, \quad \sH_{1,\{-2, -1\}}\simeq \sM_{1,\{-2,-1\}}.$$
\end{thm}

The equivalence $\sH_1\simeq \sM_1$ 
is given in \ref{HtoM} and \ref{MtoH} below. 

\begin{para}\label{HtoM} 
First we define the functor $\sH_1\to \sM_1$.

 Let $H$ be an object of $\sH_1$. The corresponding object $[\sF\to G]$ of $\sM_1$ is as follows.
\begin{eqnarray*}
G & = & W_{-1}H_\Z\bs W_{-1}H_V/F^0W_{-1}H_V. \\ 
\sF_\et & = & \gr^W_0(H_\Z). \\ 
\sF_\infi & = & \text{the formal completion of}\;\Ker(\gr^W_0(H_V)\to \gr^W_0(H_\C)).
\end{eqnarray*}
Here $\sF_\et$ is the \'etale part of $\sF$ and $\sF_\infi$ is the infinitesimal part of $\sF$. The homomorphism $\sF=\sF_\et\oplus \sF_\infi\to G$ is given as follows. 

The part $\sF_\et\to G$: Let $x\in \sF_\et=\gr^W_0H_\Z$. Since the sequence $0\to W_{-1}H_\Z\to H_\Z\to \gr^W_0H_\Z\to 0$ is exact, we can lift $x$ to an element $y$ of $H_\Z$ and this lifting is unique modulo $W_{-1}H_\Z$. Since the sequence $0\to F^0W_{-1}H_V\to F^0H_V\to F^0\gr_0^WH_V\to 0$ is exact, we can lift $x$ to an element $z$ of $F^0H_V$ and this lifting is unique modulo $F^0W_{-1}H_V$. Note that $y-z\in W_{-1}H_V$. We have a well-defined homomorphism $$\sF_\et=\gr^W_0H_\Z \lra W_{-1}H_\Z\bs W_{-1}H_V/F^0W_{-1}H_V=G\;;\;x\mapsto y-z.$$

The part $\sF_\infi\to G$: Note that $\Hom(\sF_\infi, G)$ is identified with $\Hom_\C(\Lie(\sF_\infi), \Lie(G))$. We give the corresponding homomorphism $\Lie(\sF_\infi)= \Ker(\gr^W_0(H_V) \to \gr^W_0(H_\C))\to \Lie(G)=W_{-1}H_V/F^0W_{-1}H_V$. 
 Let $x\in \Ker(\gr^W_0(H_V)\to \gr^W_0(H_\C))$. The given splitting of the weight filtration on $\Ker(H_V\to  H_\C)$ sends $x$ to an element $y$ of $\Ker(H_V\to  H_\C)$.
 Since the sequence $0\to F^0W_{-1}H_V\to F^0H_V\to F^0\gr^W_0H_V\to 0$ is exact, we can lift $x$ to an element $z$ of $F^0H_V$ and this lifting is unique modulo $F^0W_{-1}H_V$. Note that $y-z\in W_{-1}H_V$. We have a well-defined homomorphism $$\Ker(\gr^W_0H_V \to  \gr^W_0H_\C) \lra W_{-1}H_V/F^0W_{-1}H_V=\Lie(G)\;;\;x \mapsto y-z.$$

\end{para}

\begin{para}\label{MtoH} We give the functor
$\sM_1\to \sH_1$.

 Let $[\sF\to G]$ be an object of $\sM_1$. The corresponding object $H$ of $\sH_1$ is as follows. Let $0\to L \to G \to A \to 0$ be the exact sequence of 
 commutative algebraic groups  where $A$ is an abelian variety and $L$ is affine. Let $\sF_\et$ be the \'etale part of $\sF$ and let $\sF_\infi$ be the infinitesimal part of $\sF$. 

First, 
$H_\Z$ is the fiber product of $\sF_\et \to G \leftarrow \text{Lie}(G)$ where $\text{Lie}(G)\to G$ is the exponential map so we have a commutative diagram of exact sequences 
$$\begin{matrix} 0 & \to & H_1(G, \Z) & \to & H_\Z & \to & \sF_\et & \to & 
0\\
& & \downarrow & & \downarrow & & \downarrow & & \\
0 & \to & H_1(G, \Z)& \to & \text{Lie}(G) & \to & G & \to & 0.\end{matrix} $$
The weight filtration on $H_\Z$ is given as follows. $$W_0H_\Z=H_\Z, \quad
W_{-1}H_\Z= H_1(G, \Z), $$ $$W_{-2}H_\Z=H_1(L, \Z)=\text{Ker}(H_1(G, \Z)\to H_1(A, \Z)), \quad W_{-3}H_\Z=0.$$

Next, $$H_V= H_\C \oplus \Lie(L_a)\oplus \Lie(\sF_\infi)$$ where
$L_a$ is the additive part of $L$. 
The weight filtration on $H_V$ is as follows. $$W_0H_V=H_V, \quad
W_{-1}H_V= H_1(G, \C)\oplus \Lie(L_a), $$ $$W_{-2}H_V=H_1(L, \C) \oplus \Lie(L_a), \
 \quad W_{-3}H_V=0.$$

The splitting of the weight filtration on $\Ker(H_V\to  H_\C)=\Lie(L_a)\oplus \Lie(\sF_\infi)$ is defined to be this direct decomposition.

The Hodge filtration on $H_V$ is given as follows. 
$$F^{-1}H_V=H_V, \quad F^1H_V=0,$$
$$F^0H_V= \text{Ker}(H_V\to \text{Lie}(G))$$
where $H_V\to \text{Lie}(G)$ is defined as follows. The part
$ H_\C\to \Lie(G)$ of it is the $\C$-linear map induced by the canonical map $H_\Z\to \Lie(G)$. The part 
$\Lie(L_a)\to \Lie(G)$ of it is the inclusion map. The part $\Lie(\sF_\infi)\to \Lie(G)$ of it is the homomorphism induced by $\sF_\infi\to G$. We have hence
$H_V/F^0H_V\cong \Lie(G)$.

It is easy to see that this functor $\sM_1\to \sH_1$ is quasi-inverse to the functor $\sH_1\to \sM_1$ in \ref{HtoM}. 
\end{para}

\begin{para}\label{ext1} The induced functor $\sH_{1, \{-1,-2\}}\overset{\simeq}\to \sM_{1, \{-1,-2\}}$ is especially simple. It is
$$H\longmapsto [0 \to H_\Z\bs H_V/F^0H_V].$$

\end{para}

\section{Generalized Albanese varieties}
\label{sec:Alb_with_mod}

Let $k$ be an algebraically closed field of characteristic $0$ and let $X$ 
be a proper smooth algebraic variety over $k$ of dimension $n$.  We review generalized Albanese varieties $\Alb_{\sF}(X)$ defined in 
\cite{Ru}\footnote{In \cite{Ru}, X was assumed to be projective. 
This assumption was used only for singular X, which is not our concern 
here. The construction of the $\Alb_{\sF}(X)$ is valid in the same way 
for proper X. }. 
For an effective divisor $Y$ on $X$, the 
generalized Albanese variety $\Alb(X, Y)$ of  modulus $Y$ is 
a special case of $\Alb_{\sF}(X)$.

The Albanese variety $\Alb(X)$ is defined by a universal mapping 
property for morphisms from $X$ 
to abelian varieties. 
Similarly, the generalized Albanese variety $\Alb(X, Y)$ of modulus $Y$ is characterized by a universal property for 
morphisms from $X-Y$ into commutative algebraic groups with ``modulus'' $\leq Y$. 
See Prop.~\ref{propmodulus}.

\begin{para} Let $\ul{\Div}_X$ be the sheaf of abelian groups on $\Ab{k}$ defined as follows. For any commutative ring $R$ over $k$, 
$\ul{\Div}_X(R)$ is the group of all Cartier divisors on $X\otimes_k R$
 generated locally on $\Spec(R)$ by effective Cartier divisors which are flat over $R$. Let $\ul{\Pic}_X$ be the Picard functor, and let $\ul{\Pic}^0_X\subset  \ul{\Pic}_X$ be the Picard variety of $X$. 
We have the class map  $ \ul{\Div}_X\to \ul{\Pic}_X$. Let $\ul{\Div}^0_X\subset \ul{\Div}_X$ be the inverse image of $\ul{\Pic}^0_X$.

\end{para}

\begin{para} Let $\Lambda$ be the set of all 
 subgroup sheaves $\sF$ of $\ul{\Div}^0_X$ such that $\sF\cong \Z^t \oplus (\widehat \G_a)^s$ for some $t$ and $s$. For $\sF\in \Lambda$,  we have an object
$[\sF\to \ul{\Pic}^0_X]$ of
 $\sM_{1, \{0,-1\}}$. The generalized Albanese variety 
 $\Alb_{\sF}(X)$ is defined in \cite{Ru} to be the Cartier dual of 
 $[\sF\to \ul{\Pic}^0_X]$. It is an object of $\sM_{1, \{-1,-2\}}$ and hence is a commutative connected algebraic group over $k$.
 
 If $\sF, \sF'\in \Lambda$ and $\sF\subset \sF'$,  we have a canonical surjective homomorphism $\Alb_{\sF'}(X) \to \Alb_{\sF}(X)$. In the case $\sF=0$, $\Alb_{\sF}(X)=\Alb(X)$. 

\end{para}

\begin{para} \label{defAlb(X,Y)}
Let $Y$ be an effective divisor of $X$. Then the generalized Albanese variety with modulus $Y$ is defined as $\Alb_{\sF}(X)$ where $\sF=\sF_{X,Y}\in \Lambda$ is defined as follows. 
The \'etale part $\sF_\et$ of $\sF$ is the subgroup of $\ul{\Div}^0_X(k)$
consisting of all divisors whose support is contained in the support of $Y$.
The infinitesimal part $\sF_{\infi}$ of $\sF$ is as follows. Let $I$ be the ideal of $\sO_X$ (though the notation $\sO_X$ is often used in this paper for the sheaf of analytic functions, $\sO_X$ here stands for the usual algebraic object  on the Zariski site) defining $Y$, let $I_1$ be the ideal of $\sO_X$ which defines the reduced part of $Y$, and let $J=II_1^{-1}\subset \sO_X$. Then $\sF_{\infi}$ is the formal completion $\widehat \G_a\otimes_k H^0(X, J^{-1}/\sO_X)$ of the 
finite dimensional $k$-vector space $H^0(X, J^{-1}/\sO_X)$,  which is embedded in
$\ul{\Div}^0_X$
by the exponential map
$$\exp\;:\; \widehat \G_a\otimes_k H^0(X, J^{-1}/\sO_X)\to \ul{\Div}^0_X.$$

 If $Y'$ is an effective divisor on $X$ such that $Y'\geq Y$, then $\sF_{X, Y}\subset \sF_{X, Y'}$ and hence we have a canonical surjective homomorphism $\Alb(X, Y') \to \Alb(X, Y)$. In the case $Y=0$, $\Alb(X, Y)=\Alb(X)$. 

In the case $X$ is a curve, $\Alb(X, Y)$ coincides with the generalized Jacobian variety $J(X, Y)$ of $X$ with modulus $Y$ as is explained in \cite[Exm.\ 2.34]{Ru}. 

\end{para}

\begin{para}\label{pole} As in \cite{Ru}, for $\sF\in \Lambda$, we have a rational map 
$$\alpha_{\sF}: X\to \Alb_{\sF}(X)$$
which is canonically defined up to translation by a $k$-rational point of $\Alb_\sF(X)$. 
If $\sF'\in \Lambda$ and $\sF\subset \sF'$, then $\alpha_{\sF}$ and $\alpha_{\sF'}$ are compatible via the canonical surjection 
$\Alb_{\sF'}(X) \to \Alb_{\sF}(X)$. 

For an effective divisor $Y$ on $X$, we denote the rational map $\alpha_{\sF_{X,Y}}$ 
simply by $\alpha_{X,Y}$. 
In Prop.~\ref{propmodulus} (2) below, we give a universal property of $\alpha_{X,Y}: X\to\Alb(X, Y)$ concerning rational maps from $X$ to commutative algebraic groups. This 
property follows from a general universal property of  $\alpha_{\sF}: X\to \Alb_{\sF}(X)$ obtained in \cite{Ru}, as is shown in \ref{rational} below. 

\end{para}

\begin{para} \label{defmodulus}
Let $G$ be a commutative connected algebraic group over $k$ and let $\phe: X\to G$ be a rational map. 
We define an effective divisor $\modu(\phe)$ on $X$ which we call the modulus of $\phe$.

We treat $X$ as a scheme. This divisor $\modu(\phe)$ is written in the form
$\sum_v \modu_v(\phe)v$, where $v$ ranges over all points of $X$ of codimension one and $\modu_v(\phe)$ is a non-negative integer defined as follows.

Let $0\to L\to G \to A\to 0$ be the canonical decomposition of $G$ and take an isomorphism 

\medskip

(1) \quad $L_a \cong (\G_a)^s$

\medskip 
\noindent
where $L_a$ is the additive part of $L$.

Let $K$ be the function field of $X$, and regard $\phe$ as an element of $G(K)$. 
 Since the local ring $\sO_{X, v}$ of $X$ at $v$ is a discrete valuation ring and since $A$ is proper, we have $A(O_{X, v})=A(K)$. By the commutative diagram of exact sequences
 $$\begin{matrix} 0 & \to & L(O_{X, v}) & \to & G(\sO_{X, v}) & \to & A(O_{X, v}) & \to & 0 \\
 && \cap && \cap && \Vert && \\
 0 & \to & L(K) & \to & G(K) & \to & A(K) & \to & 0, \end{matrix}$$
 we have
$G(K)=L(K)G(O_{X, v})$. Write $\phe\in G(K)$ as 

\medskip

(2) \quad $\phe= lg$ \; with $l\in L(K)$ and $g\in G(O_{X, v})$, 

\medskip
\noindent
Let $(l_j)_{1\leq j\leq s}$ 
 be the image of $l$ in 
 $(\G_a)^s(K)$. 
 
 If $\phe$ belongs to $G(O_{X, v})$, we define $\modu_v(\phe)=0$.
 Assume $\phe$ does not belong to $G(O_{X,v})$. Then we define
 $$\modu_v(\phe)= 1+ \max(\{-\text{ord}_{v}(l_j)\;|\;1\leq j\leq s\}\cup\{0\}).$$ 
This integer $\modu_v(\phe)$ is independent of the choice of the isomorphism (1) and of the choice of the presentation (2) of $\phe$. 
 
For example, if $G=\G_m$, $\modu_v(\phe)$ is $0$  if the element $\phe$ of $G(K)=K^\times$ belongs to $\sO_{X,v}^\times$, and is $1$ otherwise. If $G=\G_a$, 
$\modu_v(\phe)$ is $0$  if the element $\phe$ of $G(K)=K$ belongs to $\sO_{X,v}$, and is $m+1$ if $\phe$ has a pole of order $m\geq 1$ at $v$.

\end{para}

\begin{prop}\label{propmodulus} 
Let $G$ be a commutative connected algebraic group over $k$ and 
let $\phe : X\to G$ be a rational map. 

\medskip

{\rm(1)} For a dense open set $U$ of $X$, $\phe$ induces a morphism $U\to G$ (not only a rational map) if and only if the support of $\modu(\phe)$ does not 
meet $U$. 

\medskip

{\rm(2)} Let $Y$ be an effective divisor on $X$. Then the following two conditions (i) and (ii) are equivalent.

\medskip

\qquad {\rm (i)}  There is a homomorphism $h: \Alb(X, Y) \to G$ such that $\phe$ coincides with $h \circ \alpha_{X,Y}$ modulo a translation by $G(k)$.

\medskip

\qquad {\rm(ii)} \:\;
$\modu(\phe) \leq Y.$

\medskip 

Furthermore, if these equivalent conditions are satisfied, such homomorphism $h$ is unique.
\end{prop}

It is easy to prove (1). The proof of (2) is given in \ref{proppf} below after we review results on $\Alb_{\sF}(X)$ from \cite{Ru}.



\begin{para}\label{rational}  We review a general universal property of $\Alb_{\sF}(X)$  proved in \cite{Ru} concerning rational maps from $X$ into commutative algebraic groups.

Let $\phe:X\to G$ be a rational map into a commutative connected algebraic group $G$, 
and let $L$ be the canonical connected affine subgroup such that the 
quotient $G/L$ is an abelian variety. 
One observes that $\phe$ induces a natural transformation 
$\tau_\phe:L^\vee \to \ul{\Div}^0_X$ (see \cite[section 2.2]{Ru}), 
where $L^\vee = \ul{\Hom}_{\Ab{k}}(L,\G_m)$ is the Cartier dual of $L$.
It is shown in \cite[section 2.3]{Ru} that if $\sF\in \Lambda$, 
there is a rational map $\alpha_{\sF}: X\to \Alb_{\sF}(X)$ for which the corresponding homomorphism $\tau_{\alpha_{\sF}}: \sF\to \ul{\Div}^0_X$ coincides with the inclusion map, and such rational map $\alpha_{\sF}$ is unique up to translation by a $k$-rational point of $\Alb_{\sF}(X)$. 
For a rational map $\phe:X\to G$ into a commutative connected algebraic group $G$ and for $\sF\in \Lambda$, there is a homomorphism $h:\Alb_{\sF}(X) \to G$ such that $f$ coincides with $h\circ \alpha_{\sF}$ up to a translation by $G(k)$ if and only if the image of the homomorphism $\tau_\phe:L^\vee \to \ul{\Div}^0_X$  is contained in $\sF$. Furthermore, if such $h$ exists, it is unique. 

Moreover,  any rational map $\phe:X\to G$ into a commutative connected algebraic group $G$ coincides with $h\circ \alpha_{\sF}$ up to a translation by $G(k)$ for some $\sF\in \Lambda$ and for some homomorphism $h: \Alb_{\sF}(X)\to G$. This is because there is always some $\sF\in \Lambda$ which contains the image of
$L^{\vee}\to \ul{\Div}^0_X$. 
\end{para}

\begin{para}\label{proppf}  
We prove Prop.~\ref{propmodulus}.  
By \ref{rational} we find that  condition (i) of \ref{propmodulus} (2) is equivalent to 

\qquad (i') \quad $\im(\tau_\phe) \subset \sF_{X,Y}$. \\
Write $$Y=\sum_v e_v v$$
where $v$ ranges over all points of $X$ of codimension one and $e_v\in \N$. 
Condition (ii) of \ref{propmodulus} (2) is expressed as 

\qquad (ii') \quad  $\modu_v(\phe) \leq e_v$ 
for all points $v$ of codimension one in $X$. \\
Fix an isomorphism $L \cong (\G_m)^t\times (\G_a)^s$.
 For each point
$v$ of $X$ of codimension one, take a presentation $\phe=lg$ as in (2) in \ref{defmodulus}, let $(l'_{v,j})_{1\leq j\leq t}$ be the image of $l$ in $(\G_m)^t(K)=(K^\times)^t$, and as in \ref{defmodulus}, let $(l_{v,j})_{1\leq j\leq s}$ be the image of $l$ in $(\G_a)^s(K)=K^s$. Note that

(a) \quad $\phe\in G(\sO_{X,v})$ if and only if $l'_{v,j}\in \sO_{X,v}^\times$ for $1\leq j\leq t$ and $l_{v,j}\in \sO_{X, v}$ for $1\leq j\leq s$.\\
By construction of the transformation $\tau_\phe$ 
in \cite[section 2.2]{Ru},  we have the following (b) and (c).

(b) The \'etale part of $\tau_{\phe}$ 
$$\tau_{\phe,\text{\'et}}: \Z^t\to \ul{\Div}^0_X(k)$$
sends the $j$-th base of $\Z^t$ ($1\leq j\leq t$) to 
the divisor $\sum_v \text{ord}_v(l'_{v,j})v$. 

\medskip

(c) The infinitesimal part of $\tau_{\phe}$ 
$$\tau_{\phe,\text{inf}}: (\widehat \G_a)^s\to \ul{\Div}^0_X$$
has the form
$$(a_j)_{1\leq j\leq s}\mapsto \exp\left(\sum_{j=1}^s a_jf_j\right)$$
for some $f_j\in \Gamma(X, K/\sO_X)=\text{Lie}(\ul{\Div}^0_X)$ ($1\leq j\leq s$) such that for any point $v$ of $X$ of codimension one, the stalk of $f_j$ at $v$ coincides with $l_{v,j}\mod~\sO_{X,v}$. 

Condition (i') is equivalent to the condition that the following (i'$_{\text{\'et}}$) and (i'$_{\text{inf}}$) are satisfied.

\begin{tabular}{ll} 
\qquad (i'$_{\text{\'et}}$) & The image of $\tau_{\phe,\text{\'et}}$ is contained in the \'etale part of $\sF_{X,Y}$. \\ 
\qquad (i'$_{\text{inf}}$) & The image of $\tau_{\phe,\text{inf}}$ is contained in the infinitesimal part of $\sF_{X,Y}$. 
\end{tabular} \\
By the above (b),  (i'$_{\text{\'et}}$) is equivalent to the condition that the following (i'$_{\text{\'et},v}$) is satisfied for any point $v$ of $X$ of codimension one.

\begin{tabular}{ll} 
\qquad (i'$_{\text{\'et},v}$) & If $e_v=0$, then $l'_{v,j}\in \sO_{X,v}^\times$ for $1\leq j\leq t$. 
\end{tabular} \\
On the other hand, by the above (c), (i'$_{\text{inf}}$) is equivalent to 

\hspace{70pt} $f_j\in \Gamma(X, J^{-1}/\sO_X)$ for $1\leq j\leq s$,\\
and hence equivalent to the condition that the following (i'$_{\text{inf},v}$) is satisfied for any point $v$ of $X$ of codimension one.

\begin{tabular}{ll} 
\qquad (i'$_{\text{inf},v}$) &  If $e_v=0$, then $l_{v,j}\in \sO_{X,v}$ for $1\leq j\leq s$. \\ 
 & If $e_v\geq 1$, then $\text{ord}_v(l_{v,j})\geq 1-e_v$ for $1\leq j\leq s$. 
\end{tabular} \\
By the above (a), for each $v$, (i'$_{\text{\'et},v}$) and (i'$_{\text{inf},v}$)
are satisfied if and only if $\modu_v(\phe)\leq e_v$.\qed

\end{para}

\begin{cor} 
For any $\sF\in \Lambda$, there exists an effective divisor $Y$ 
such that $\sF\subset \sF_{X, Y}$. 
\end{cor}

\begin{pf} 
Let $Y=\modu(\alpha_\sF)$ be the modulus of the rational map $\alpha_\sF:X\to\Alb_\sF(X)$ 
associated to $\sF\in\Lambda$. Then $\sF\subset \sF_{X, Y}$. 
\end{pf}

\section{Proof of Theorem~\ref{mainthm}}\label{proof}

We prove Thm.~\ref{mainthm}. 
Let $X$ be a proper smooth algebraic variety over $\C$ of dimension $n$, and let $Y$ be an effective divisor on $X$. 
Let $I$ be the ideal of $\sO_X$ which defines $Y$, let $I_1$ be the ideal of $\sO_X$ which defines the reduced part of $Y$, and let $J=II_1^{-1}\subset \sO_X$.
\medskip

\begin{para}\label{H1Hn-1} 
Let $H^1(X, Y_+)(1)$ be the object of $\sH_{1,\{0, -1\}}$ corresponding to the object $[\sF_{X, Y}\to \text{Pic}^0(X)]$ of $\sM_{1, \{0, -1\}}$ in the equivalence of categories \ref{thm2}.
Let $H^{2n-1}(X, Y_-)(n)$ be the object of $\sH_{1,\{-1, -2\}}$ corresponding to the object $\Alb(X, Y)$ of $\sM_{1, \{-1, -2\}}$.

Since the equivalence of categories in \ref{thm2} is compatible with dualities, we have
\begin{equation}\label{HdualH}
H^{2n-1}(X, Y_-)(n) \,\cong\, \ul{\Hom}(H^1(X, Y_+)(1), \Z)(1).
\end{equation}

We prove Thm.~\ref{mainthm} in the following way. First in \ref{H1}, we give an explicit description of 
$H^1(X, Y_+)(1)$. From this, by (\ref{HdualH}), 
we can obtain an explicit description of $H^{2n-1}(X, Y_-)(n)$ as in \ref{Hn-1}. Since $\Alb(X, Y)$ corresponds to $H^{2n-1}(X, Y_-)(n)$ in the equivalence of categories
$\sH_{1, \{-1, -2\}}\simeq \sM_{1, \{-1,-2\}}$, we can obtain from \ref{Hn-1}  the explicit descriptions of $\Alb(X, Y)$ stated in  Thm.~\ref{mainthm}.

 We define objects $H^1(X, Y_+)$ and $H^{2n-1}(X, Y_-)$ of $\sH$ as follows: 
 $H^1(X, Y_+)$ is the Tate twist $(H^1(X, Y_+)(1))(-1)$ of  $H^1(X, Y_+)(1)$,  and 
$H^{2n-1}(X, Y_-)$ is the Tate twist $(H^{2n-1}(X, Y_-)(n))(-n)$ of $H^{2n-1}(X, Y_+)(n)$. 
 These are natural generalizations of the objects of $\sH$ for the curve case considered in \ref{example}.

\end{para}

\begin{para}\label{maps} We define canonical $\C$-linear maps
\begin{equation}\label{map1} H^1(X-Y, \C) \lra H^1(X, \sO_X), 
\end{equation}
\begin{equation}\label{map2} H^{n-1}(X, \Omega_X^n)\lra H_c^{2n-1}(X-Y, \C)
\end{equation}

First assume that $Y$ is with normal crossings. Then
by \cite{De}, 
we have canonical isomorphisms 
$$H^m(X-Y, \C) \cong H^m(X, \Omega_X^{\bullet}(\log(Y))), 
\quad H^m_c(X-Y, \C) \cong 
H^m(X, \Omega_X^{\bullet}(-\log(Y)))$$
for $m\in \Z$, where $\Omega_X^p(\log(Y))$ is the sheaf of differential $p$-forms with log poles along $Y$, and $\Omega_X^p(-\log(Y))=I_1\Omega_X^p(\log(Y))$. Since $\sO_X=\Omega_X^0(\log(Y))$ and $\Omega^n_X=\Omega_X^n(-\log(Y))$, we have canonical maps of complexes $\Omega_X^{\bullet}(\log(Y)) \to \sO_X$ and
$\Omega_X^n[-n] \to \Omega^{\bullet}_X(-\log(Y))$. These maps induces the maps (\ref{map1}) and (\ref{map2}) in the case $Y$ is with normal crossings, respectively.

In general, take a birational morphism $X'\to X$ of proper smooth algebraic varieties over $\C$ such that the inverse image $Y'$ of $Y$ on $X'$ is with normal crossings. Then we have maps
$$H^{n-1}(X, \Omega_X^n)\lra H^{n-1}(X', \Omega_{X'}^n)\lra H^{2n-1}_c(X'-Y', \C)= H_c^{2n-1}(X-Y, \C)$$
where the second arrow is the map (\ref{map2}) for $X'$, and the composition $H^{n-1}(X, \Omega_X^n)\to 
H_c^{2n-1}(X-Y, \C)$ is independent of the choice of $X'\to X$. 
The $\C$-linear dual of (\ref{map2}) with respect to the Poincar\'e duality and Serre duality gives the map (\ref{map1}). The map (\ref{map1}) is also obtained as 
the composition
$$H^1(X-Y, \C)= H^1(X'-Y', \C) \lra H^1(X', \sO_{X'})\overset{\simeq}\longleftarrow H^1(X, \sO_X).$$

\end{para}

\begin{para}\label{H1} Let $H=H^1(X, Y_+)(1)$, the object of $\sH_{1, \{0, -1\}}$ corresponding to the object $[\sF_{X,Y}\to \Pic^0(X)]$ of $\sM_{1, \{0, -1\}}$. We describe $H$. By \cite{BaS} Theorem 4.7 which treats the case $Y$ has no multiplicity, we  can identify $H_\Z$ with $H^1(X-Y, \Z(1))$ and identify the map $H_\C\to \Lie(\Pic^0(X))=H^1(X, \sO_X)$ 
with the map (\ref{map1}) in \ref{maps}. We have $H_V=H_\C\oplus H^0(X, J^{-1}/\sO_X)$, the maps $a:H_\C\to H_V$ and $b: H_V\to H_\C$ are the evident ones, 
the weight filtration is given by $W_0H=H$, $W_{-2}H=0$,
$$W_{-1}H_\Q=H^1(X, \Q(1)), \quad W_{-1}H_V = H^1(X, \C),$$
and the Hodge filtration is given by $F^{-1}H_V=H_V$, $F^1H_V=0$, and
$$F^0H_V = \Ker(H^1(X-Y, \C) \oplus H^0(J^{-1}/\sO_X)\to H^1(X,\sO_X))$$
where the map 
$H^0(J^{-1}/\sO_X) \to H^1(X, \sO_X)$ is the connecting map of the exact sequence $0 \to \sO_X\to J^{-1}\to J^{-1}/\sO_X\to 0$.
\end{para}

\begin{para}\label{Hn-1} Let $H=H^{2n-1}(X, Y_-)(n)$, 
the object of $\sH_{1, \{-1, -2\}}$ corresponding to the object $\Alb(X, Y)$
of $\sM_{1, \{-1, -2\}}$. By  (\ref{HdualH}) in \ref{H1Hn-1}, we obtain the following description of 
$H$ from the description of $H^1(X, Y_+)(1)$ in \ref{H1}.

$$H_\Z=H^{2n-1}_c(X-Y, \Z)/(\text{torsion}), \quad H_V=H_\C\oplus H^{n-1}(X, \Omega_X^n/J\Omega_X^n),$$
the maps $a:H_\C\to H_V$ and $b: H_V\to H_\C$ are the evident ones, 
the weight filtration is given by $W_{-1}H=H$, $W_{-3}H=0$,
$$W_{-2}H_\Q=\Ker(H_\Q\to H^{2n-1}(X, \Q(n))), \quad W_{-2}H_V =\Ker(H_V\to  H^{2n-1}(X, \C)),$$
and the Hodge filtration is given by $F^{-1}H_V=H_V$, $F^1H_V=0$, and
$$F^0H_V = \text{Image}(H^{n-1}(X, \Omega_X^n) 
\lra H_c^{2n-1}(X-Y, \C)\oplus
 H^{n-1}(X, \Omega_X^n/J\Omega^n_X))$$
where the map $H^{n-1}(X, \Omega_X^n) \to H_c^{2n-1}(X-Y, \C)$ is (\ref{map2}) in \ref{maps} 
and the map $H^{n-1}(X, \Omega_X^n) \to H^{n-1}(X, \Omega_X^n/J\Omega^n_X)$ is the evident one.

\end{para}

\begin{para}  We prove Thm.~\ref{mainthm} (2). Let $H=H^{2n-1}(X, Y_-)(n)$. Then
$$\Alb(X, Y)= H_\Z\bs H_V/F^0H_V$$ by \ref{ext1}. Hence the description of $H^{2n-1}(X, Y_-)(n)$ in \ref{Hn-1} proves  Thm.~\ref{mainthm} (2).

\end{para}

\begin{para}\label{Deligne} As a preparation for the proof of Thm.~\ref{mainthm} (1), we review a kind of Serre-duality obtained in Appendix by Deligne in the book \cite{Ha}.

Let $S$ be a proper scheme over a field $k$, let $C$ be a closed subscheme of $S$, let $U=S-C$, and let $I_C$ be the ideal of $\sO_S$ which defines $C$.
Assume $U$ is smooth over $k$ and purely of dimension $n$. 
Let $\sF$ be a coherent $\sO_S$-module. Then for any $p\in \Z$, we have a canonical isomorphism
$$H^p(U, R{\sH}om_{\sO_U}(\sF|_U, \Omega_U^n)) \,\cong\, 
\varinjlim_m \Hom_k(H^{n-p}(X, I_C^m\sF), k).$$
In the case $C$ is empty and $\sF$ is locally free, this is the usual Serre duality
$$H^p(X, {\sH}om_{\sO_X}(\sF, \Omega_X^n)) \,\cong\, 
\Hom_k(H^{n-p}(X,\sF), k).$$
\end{para}

\begin{para} We start the proof of Thm.~\ref{mainthm} (1).

Let $C_Y$ be the subcomplex of $\Omega^{\bullet}_X$ defined as
$$C_Y^p=\ker(\Omega_X^p\to \Omega_Y^p) \;\;\text{for}\; 0\leq p\leq n-1, \quad C_Y^n= J\Omega^n_X.$$

\end{para}

\begin{prop}\label{propHc} For $p=2n, 2n-1$, the maps $H^{p}_c(X-Y, \C) \to H^{p}(X, C_Y)$ induced by the homomorphism $j_!\C\to C_Y$ are isomorphisms.  
\end{prop}

\begin{para} We prove Prop.~\ref{propHc} in the case $Y=Y_1$. 
We have an exact sequence of complexes
$$0 \lra C_{Y_1} \lra \Omega_X^{\bullet}\lra \Omega_{Y_1}^{\leq n-1}\lra 0.$$
Since the support of $\Omega_{Y_1}^{\leq n-1}$ is of dimension $\leq n-1$ and since $\Omega_{Y_1}^{\leq n-1}$ has only terms of degree $\leq n-1$, we have $H^p(X, \Omega_{Y_1}^{\leq n-1})=0$ for $p\geq 2n-1$. Hence 
$$H^{2n}(X, C_{Y_1})\cong H^{2n}(X, \Omega_X^{\bullet})\cong H^{2n}(X, \C) \cong H^{2n}_c(X-Y, \C).$$ The above exact sequence of complexes induces the lower row of the
commutative diagram of exact sequences
$$\begin{matrix} 
H^{2n-2}(X, \C)& \to & H^{2n-2}(Y_1, \C) & \to & H^{2n-1}_c(X-Y_1, \C) &\to &  H^{2n-1}(X, \C) & \to 0 \phantom{.} \\ 
\downarrow & & \downarrow &  & \downarrow & &  \downarrow & & \\ 
H^{2n-2}(X, \Omega^{\bullet}_X)& \to & H^{2n-2}(Y_1, \Omega_{Y_1}^{\leq n-1}) & \to & H^{2n-1}(X, C_{Y_1}) &\to &  H^{2n-1}(X, \Omega_X^{\bullet}) & \to 0.\end{matrix}$$
The vertical arrows except possibly the map $H^{2n-1}_c(X-Y_1, \C) \to H^{2n-1}(X, C_{Y_1})$ are isomorphisms. Hence the last map is also an isomorphism. 
\end{para}

\begin{lem}\label{Y'Y''} Let $Y'$ and $Y''$ be effective divisors on $X$ whose supports coincide with $Y_1$ and assume $ Y'\geq Y''$. Then the canonical map $H^{2n-1}(X, C_{Y'})\to H^{2n-1}(X, C_{Y''})$ is surjective and the canonical map $H^{2n}(X, C_{Y'})\to H^{2n}(X, C_{Y''})$ is an isomorphism.
\end{lem} 

\begin{proof} Let $N=C_{Y''}/C_{Y'}$. We have $$N^p=\Ker(\Omega_{Y'}^p\to \Omega_{Y''}^p) \;\;\;\text{for}\;0\leq p\leq n-1, \quad N^n= J''\Omega_X^n/J'\Omega_X^n.$$ Here, $J'=I'I_1^{-1}$, $J''=I''I_1^{-1}$ with $I'$ (resp. $I''$) the ideal of $\sO_X$ which defines $Y'$ (resp. $Y'')$. Since the support of $N$ is of dimension $\leq n-1$  and $N$ has only terms of degree $\leq n$, we have $H^{2n}(X, N)=0$. Hence it is sufficient to prove $H^{2n-1}(X, N)=0$.

Let $\Sigma$ be the set of all singular points of $Y_1$. Then $\Sigma$ 
is of dimension $\leq n-2$. Let $\Omega^{\bullet}_X(\log(Y_1))$ be the de Rham complex on $X-\Sigma$ with log poles along $Y_1-\Sigma$. Then 
as is easily seen, the restriction of $C_Y$ to $X-\Sigma$ coincides with $I\Omega_X^{\bullet}(\log(Y_1))$. Let $I_{\Sigma}$ be the ideal of $\sO_X$ defining $\Sigma$ (here $\Sigma$ is endowed with the reduced structure). For $k\geq 0$, let $N_k$ be the subcomplex of $N$ defined by 
$N_k^p=I_{\Sigma}^{\max(k-p, 0)}N^p$. In particular, $N_0=N$. Then if 
$k\geq j \geq 0$, since the support of $N_j/N_k$ is of 
dimension $\leq n-2$ and $N_j/N_k$ has only terms of degree $\leq n$, we have $H^{2n-1}(X, N_j/N_k)=0$. Hence $H^{2n-1}(X, N_k) \to H^{2n-1}(X, N_j)$ is surjective. 
Applying \ref{Deligne} for 
$S=X$ and $C=\Sigma$ 
yields that  $\varprojlim_k H^{2n-1}(X, N_k)$ 
is the dual vector space of  $H^0(X-\Sigma, 
 [(J')^{-1}/(J'')^{-1}\overset{d}\longrightarrow  (J')^{-1}\Omega_X(\log Y_1)/(J'')^{-1}\Omega_X(\log Y_1)])$.  \linebreak 
Since $d: (J')^{-1}/(J'')^{-1}
\to  (J')^{-1}\Omega_X(\log Y_1)/(J'')^{-1}\Omega_X(\log Y_1)$ is injective, the last cohomology group is $0$. Hence $H^{2n-1}(X, N_k)=0$ for all $k\geq 0$. In particular, $H^{2n-1}(X, N)=0$. 
\end{proof} 

\begin{para} We prove Prop.~\ref{propHc} in general.
 By Lemma \ref{Y'Y''}, the map $\varprojlim_{Y'} H^{2n-1}(X, C_{Y'})\to H^{2n-1}(X, C_Y)$ is surjective, where $Y'$ ranges over all effective divisors on $X$ whose supports coincide with $Y_1$. By \ref{Deligne} which we apply by taking $S=X$ and $C=Y$, we have that
  $\varprojlim_{Y'} H^{2n-1}(X, C_{Y'})$ is the dual vector space of $H^1((X-Y)_{\text{zar}}, \Omega^{\bullet}_{X-Y, \text{alg}})$ where zar means Zariski topology and alg means the algebraic version. But $H^1((X-Y)_{\text{zar}}, \Omega^{\bullet}_{X-Y, \text{alg}}) \simeq H^1(X, \C)$. This proves 
$\varprojlim_{Y'} H^{2n-1}(X, C_{Y'})\cong H^{2n-1}_c(X-Y, \C)$. 
Hence the map \linebreak 
$H^{2n-1}_c(X-Y, \C) \to H^{2n-1}(X, C_Y)$ is surjective. 
Since the composition
$H^{2n-1}_c(X-Y, \C) \linebreak 
\to H^{2n-1}(X, C_Y) \to H^{2n-1}(X, C_{Y_1})\cong H^{2n-1}_c(X-Y, \C)$ is the identity map, the map $H^{2n-1}_c(X-Y, \C) \to H^{2n-1}(X, C_Y)$ is an isomorphism. 
\end{para}

\begin{para}  We prove (1) of Thm.~\ref{mainthm}. Let $S_Y=\Ker(\Omega^{\bullet}_X \to \Omega^{\leq n-1}_Y)$. Then $C_Y\subset S_Y\subset C_{Y_1}$. We have an exact sequence of complexes
$$0 \lra C_Y\lra S_Y\lra \Omega_X^n/J\Omega_X^n[-n]\lra 0.$$
Hence we have an exact sequence
$$H^{2n-1}(X, C_Y) \to H^{2n-1}(X, S_Y) \to H^{n-1}(X, \Omega_X^n/J\Omega^n_X)
\to H^{2n}(X, C_Y) \to H^{2n}(X, S_Y).$$
Note that for $p=2n, 2n-1$, the compositions
$$H^p(X, C_Y)\lra H^p(X, S_Y)\lra H^p(X, C_{Y_1})$$
are isomorphisms by Prop.~\ref{propHc}. Hence by Prop.~\ref{propHc}, we have an isomorphism 
 $$H^{2n-1}(X, S_Y) \,\cong\, H_c^{2n-1}(X-Y, \C) \oplus H^{n-1}(X, \Omega_X^n/J\Omega^n_X)$$
which is compatible with the maps from $H^{n-1}(X, \Omega_X^n)$. Hence (1) of Thm.~\ref{mainthm} follows from (2) of Thm.~\ref{mainthm}. 
\end{para}

{\scshape
\begin{flushright}
\begin{tabular}{l}
Department of Mathematics \\  
Faculty of Science, Kyoto University \\ 
Kitashirakawa Oiwake cho, Kyoto 606-8502 \\
Japan \\
{\upshape e-mail: \texttt{kzkt@math.kyoto-u.ac.jp}}\\
{\upshape e-mail: \texttt{henrik.russell@uni-due.de}}\\
\end{tabular}
\end{flushright}
}

\end{document}